\documentclass[11pt,a4paper,twoside,reqno]{amsart}

\usepackage{amssymb}
\usepackage{latexsym}

\usepackage{graphicx}
\usepackage{a4wide}

\usepackage{amsmath}
\usepackage{amsthm}

\usepackage{color}
\usepackage{caption}
\usepackage{subcaption}
\usepackage{hyperref}
\usepackage{multirow}
\usepackage{booktabs}

\newtheorem{theorem}{Theorem}
\newtheorem{prop}{Proposition}
\newtheorem{defn}{Definition}
\newtheorem{lemma}{Lemma}

\newcommand{\norm}[1]{\left\| #1 \right\|}

\newcommand{\eps}{\varepsilon}

\renewcommand{\i}{\ifmmode\mathit{\mathchar"7010 }\else\char"10 \fi}
\renewcommand{\j}{\ifmmode\mathit{\mathchar"7011 }\else\char"11 \fi}
\newcommand{\R}{\mathbb{R}}
\newcommand{\N}{\mathbb{N}}

\begin{document}\large

\title[A filtered Chebyshev spectral method for conservation laws on network]{A filtered Chebyshev spectral method for conservation laws on network}


\author[S. F. Pellegrino]{Sabrina Francesca Pellegrino$^{\ast}$
} 



\address[S. F. Pellegrino]{Dipartimento di Ingegneria Elettrica e dell'Informazione, Politecnico di Bari, Via E. Orabona, 4 - 70125 Bari (BA), Italy}
\email{sabrinafrancesca.pellegrino@poliba.it}

\thanks{$^\ast$Corresponding author.}

\begin{abstract}

We propose a spectral method based on the implementation of Chebyshev polynomials to study a model of conservation laws on network. We avoid the Gibbs phenomenon near shock discontinuities by implementing a filter in the frequency space in order to add local viscosity able to contrast the spurious oscillations appearing in the profile of the solution and we prove the convergence of the semi-discrete method by using the compensated compactness theorem. thanks to several simulation, we make  a comparison between the implementation of the proposed method with a first order finite volume scheme.
\end{abstract}

\maketitle





\section{Introduction}

Modeling transport on networks is an important issue arising several contexts, such as for managing crowd evacuations in case of natural disasters (see~\cite{ADR-M3AS}), to study the evolution of cracks in peridynamics and in fluid-dynamic models, (see for instance~\cite{BDFP,AC2020}), to investigate infiltration in soils (see~\cite{Berardi_Difonzo_EFMC_2019,Berardi_Difonzo_Lopez_CAMWA_2020,Berardi_Difonzo_Vurro_Lopez_ADWR_2018}) and for traffic management~\cite{DSDPR,ADRR-numerical}, etc. 

Several mathematical models have been proposed to deepen such issue, and, in particular, models consisting of systems of conservation laws are often used to describe such situations. They are characterized by the fact that their solutions may develop discontinuities in finite time even when the initial conditions are analytical (see~\cite{BressanBook}). Moreover, the problem is further complicated by the fact that, when conservation laws are applied on network, there is the need to establish a suitable transmission condition at the junction, (see~\cite{GaravelloPiccoliBook}).

From a numerical point of view, systems of conservation laws are treated by using finite volume schemes, Glimm schemes (see~\cite{AGS,Pellegrino,ACD,Lebacque}) or also mimetic finite difference methods (see~\cite{Lopez_Vacca_2016}). They are first order monotone, consistent and stable schemes. A different approach consists in considering high-order spectral schemes. Indeed they are efficient method with exponential accuracy when applied to problems whose solutions have a certain amount of regularity. When solutions develop discontinuities in the domain, the exponential accuracy is lost and spectral methods show the Gibbs phenomenon, that is the formation of spurious oscillations in the profile of the solution. Such phenomenon can appear both at the boundaries of the domain or close to discontinuities. In the first case, one can overcome the phenomenon by choosing non-uniform mesh grid, with collocation points denser at the boundaries or by imposing periodic boundary conditions (see for instance~\cite{LP,LP2021,Jafarzadeh}).

Other strategies have to be adopted in order to recover global high-order accuracy when damping appears in proximity of shocks. This represents an interesting challenge and a good framework for a machine learning approach to the problem. Indeed, in~\cite{deepray}, the authors propose a projection-based reduced order model based on the Fourier collocation method for compressible flows and use a neural network to limit oscillations.


A different approach to recover spectral accuracy and to remove Gibbs phenomenon is based on the introduction of dissipative or filtered mechanism. Following this idea, in this paper we proposed a filtered Chebyshev spectral method to study numerically a system of conservation laws on network. The choice of using Chebyshev polynomials, instead of the Fourier ones, is related to the fact that Chebyshev interpolation does not require any periodicity assumption at the boundaries, and is also able to avoid Gibbs phenomenon at the boundaries thanks to the use of Chebyshev Gauss-Lobatto points. We filter the Chebyshev modes in the frequency space to reduce oscillations near discontinuities. We couple the method with a monotone transmission condition at the junction and we reformulate the methods by adding a super spectral viscosity in order to prove the convergence of the semi-discrete method, and, thanks to several simulations we show the gain of one order of convergence with respect to the implementation of a finite volume scheme.

The paper is organized as follows. In Section~\ref{sec:conslawnetwork} we state the problem and provide a briefly review on conservation laws on network. In Section~\ref{sec:Chebyshev}, we describe the filtered Chebyshev collocation method and prove its convergence. In Section~\ref{sec:fully} we construct the fully discrete method . Section~\ref{sec:numericalSimulations} is devoted to the numerical simulations. Finally, Section~\ref{sec:concl} concludes the paper.

\section{Statement of the problem}\label{sec:conslawnetwork}

We consider a model of conservation laws describing the flow of a mass density on a star shaped network consisting of two incoming edges, one outgoing edge and one junction. The generalization to network consisting of more edges is trivial. Moreover, due to the property of finite speed of propagation in hyperbolic equations, it is enough to consider networks having a unique junction.

Each edge is parameterized by $x\in\Omega_h\subset\R$, for $h\in\{1,2,3\}$, so that the network is given by $\Gamma=\Pi_{h=1}^3 \Omega_h\subset\R^3$ and the junction belongs to $\bigcap_{h=1}^3\partial \Omega_h$, (see Figure~\ref{fig:merge} for a basic representation of the network). Let denote by $\bar{x}$ the location of the junction.

\begin{figure}
    \centering
    \includegraphics[width=0.35\textwidth]{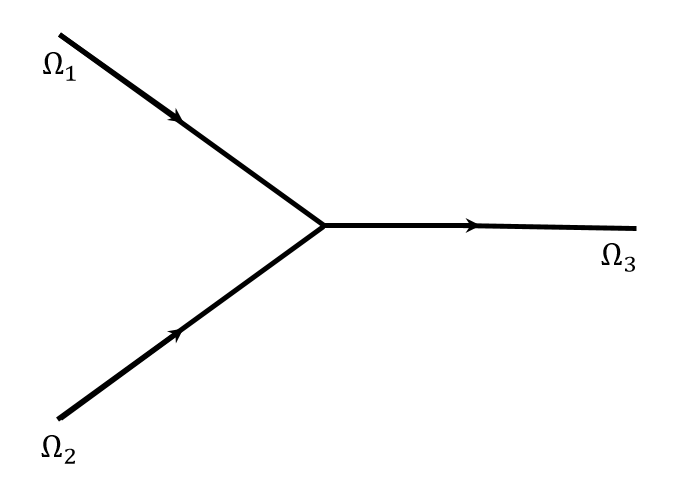}
    \caption{The representation of a network consisting in two incoming edges and one ougoing edge.}
    \label{fig:merge}
\end{figure}

On each edge, we consider an hyperbolic conservation law of the form
\begin{equation}
\label{eq:cl}
\frac{\partial u_h}{\partial t}(x,t) + \frac{\partial }{\partial x} f_h(u_h(x,t)) =0,\qquad t>0,\quad x\in\Omega_h,\quad h=1,2,3,
\end{equation}
where $u_h$ is the density and $f_h$ is the flux on each edge. We assume that the edges have a common maximal density $u_{\max}>0$ and the fluxes are bell-shaped, Lipschitz and nonlinear non-degenerate, namely we require that $f_h$ satisfies the following properties
\begin{itemize}
    \item[1.] $f_h\in Lip([0,u_{\max}];\R_+)$, with $\norm{f_h}_\infty\le L_h$, for certain $L_h>0$,
    \item[2.] $f_h(0)=f_h(u_{\max})=0$,
    \item[3.] there exists $u_{h,c}\in(0,u_{\max})$ such that $f_h'(u)\left(u_{h,c}-u\right)>0$ for a.e. $u\in[0,u_{\max}]$,
    \item[4.] $f_h'$ is not constant on any non-trivial subinterval of $[0,u_{\max}]$, 
\end{itemize}
for every $h=1,2,3$.

Under such assumptions on the flux, we have the existence of the strong traces for the solution (see~\cite{panov,vasseur}).

To complete the model, we add the initial conditions
\begin{equation}
\label{eq:initcond}
u_h(x,0)=u_{h,0}(x),\quad x\in\Omega_h,\qquad h=1,2,3,
\end{equation}
and we assume that $u_{h,0}\in L^2(\Omega_h)\cap L^\infty(\Omega_h)$, for $h=1,2,3$.

We also require the conservation of the total density at the junction, namely for a.e. t>0, it holds
\begin{equation}
\label{eq:consjunc}
f_1\left(u_1\left(\bar{x}^-,t\right)\right)\ + f_2\left(u_2\left(\bar{x}^-,t\right)\right) = f_3\left(u_3(\bar{x}^+,t)\right).
\end{equation}

The well-posedness of the Cauchy problem is proved in~\cite{ACD} in the framework of admissible solution at the junction in terms of vanishing viscosity germ, (see~\cite{AC2015,AKR2010,AKR2011}). In their work, the authors propose a finite volume scheme to approximate numerically the solution on the network and prove its convergence. While a validation of such a scheme can be found in~\cite{Pellegrino}.

The construction of an admissible weak solution is related with the definition of a Riemann solver at the junction, that is an operator which associates to any constant initial condition on the whole network a self-similar weak solution of~\eqref{eq:cl}-\eqref{eq:initcond}.

Following~\cite{ACD,Pellegrino}, we find that the selection of an admissible Riemann solver at the junction can be understood as a collection of initial boundary value problems of the form
\begin{equation}
\label{eq:IBVP}
\begin{cases}
    \frac{\partial }{\partial t} u + \frac{\partial}{\partial x}f_h\left(u_h\right)=0&\qquad (x,t)\in\Omega_h\times [0,T],\\
    u_h(x,0)=u_{h,0}(x)&\qquad x\in\Omega_h\\
    u_h(\bar{x},t)=u_{h,b}(t),&\qquad t\in[0,T],
\end{cases}\qquad h=1,2,3,
\end{equation}
where the boundary functions have to be fixed in order to ensure the conservation at the junction~\eqref{eq:consjunc}.

In~\cite{Bardos}, Bardos-LeRoux-N\'ed\'elec reformulate the conservation condition~\eqref{eq:consjunc} in terms of traces of the solution and by using the Godunov flux.

\begin{defn}
The {\em Godunov flux} associated to the flux $f_h$ is $G_h:[0,u_{\max}]^2\to\R$ defined as follows
\begin{equation}
\label{eq:Godunovflux}
G_h(a,b)=\begin{cases}
    \min_{s\in[a,b]} f_h(s),&\quad a\le b\\
    \max_{s\in[a,b]} f_h(s),&\quad a\ge b,
\end{cases}
\end{equation}
for any couple $(a,b)\in[0,u_{\max}]^2$.
\end{defn}
The Godunov flux is the function which associates to any couple $(a,b)\in[0,u_{\max}]^2$ the value $f_h(u_h(\bar{x}^-,t))=f_h(u_h(\bar{x}^+,t))$, where $u_h$ is the Kruzhkov entropy solution to the Riemann problem (see~\cite{Kruzkov,LeFloch,Bardos})
\begin{equation}
\label{eq:Riemannpb}
\begin{cases}
    \frac{\partial}{\partial t} u_h + \frac{\partial}{\partial x} f_h(u_h) = 0,&\quad x\in\R,\, t>0,\\
    u_h(x,0)=\begin{cases}
        a,\quad x<\bar{x},\\
        b,\quad x>\bar{x},
    \end{cases}
    &\quad x\in\R.
\end{cases}
\end{equation}

As proved in~\cite{Bardos}, the conservation condition~\eqref{eq:consjunc} is equivalent to ask
\begin{equation}
\label{eq:BCleft}
f_h\left(u_h\left(\bar{x}^-,t\right)\right)=G_h\left(u_h\left(\bar{x}^-,t\right),\ u_{h,b}(t)\right),\quad h=1,2,
\end{equation}
and
\begin{equation}
\label{eq:BCright}
f_h\left(u_h\left(\bar{x}^+,t\right)\right)=G_h\left(u_{h,b}(t),\ u_h\left(\bar{x}^+,t\right)\right),\quad h=3.
\end{equation}

If the system~\eqref{eq:cl} admits a solution, then there exists $u_b:\R_+\to[0,u_{\max}]$ such that (see~\cite{AC2015,AM})
\begin{equation}
\label{eq:existub}
u_{h,b}(t)=u_b(t),\qquad \text{for a.e. }t>0,\text{and for all }h=1,2,3.
\end{equation}

Using~\eqref{eq:BCleft}-\eqref{eq:BCright}-\eqref{eq:existub}, the conservation condition~\eqref{eq:consjunc} becomes
\begin{equation}
\label{eq:BCcriterion}
G_1\left(u_1\left(\bar{x}^-, t\right),\ u_b(t)\right) + G_2\left(u_2\left(\bar{x}^-, t\right),\ u_b(t)\right) = 
G_3\left(u_b(t),\ u_3\left(\bar{x}^+, t\right)\right).
\end{equation}

As a consequence, studying the model~\eqref{eq:cl}-\eqref{eq:initcond}-\eqref{eq:consjunc} is equivalent to study
\begin{equation}
\label{eq:model}
\begin{cases}
    \frac{\partial }{\partial t} u + \frac{\partial}{\partial x}f_h\left(u_h\right)=0&\qquad (x,t)\in\Omega_h\times [0,T],\\
    u_h(x,0)=u_{h,0}(x)&\qquad x\in\Omega_h\\
    u_h(\bar{x},t)=u_{b}(t),&\qquad t\in[0,T],
\end{cases}\qquad h=1,2,3,
\end{equation}
where for every $t>0$, we have to compute the boundary value $u_b(t)$ by solving the non-linear equation~\eqref{eq:BCcriterion}.

\section{Chebyshev collocation method}\label{sec:Chebyshev}

In what follows, we construct a spectral approximation of the solution to~\eqref{eq:model}-\eqref{eq:BCcriterion} by using the discrete Chebyshev expansion.

We fix the total number of collocation points for each edge of the network $N>0$ and we assume $\Omega=[-1,1]^3$. So, the junction is located at $\bar{x}_{in}=1$ for the incoming edges and at $\bar{x}_{out}=-1$ for the outgoing one. By an affine linear transformation we can always consider the product of more general intervals.

We discretize each interval by the Chebyshev-Gauss-Lobatto points (CGL) defined as $x_j=\cos\left(\frac{\pi j}{N}\right)$, for $j=0,\dots,N$. They are non uniform mesh points, simmetrically distibuted with respect the average value of the reference domain. Due to their geometry, CGL points allows us to avoid Gibb's phenomenon at the boundaries. However, as we will see in the next section, in order to remove such phenomenon near discontinuities, we need to introduce a filtered procedure.

On each edge, we seek an approximation of the solution $u_h$ given by
\begin{equation}
\label{eq:Chebyapprox}
u_h^N(x,t)=\sum_{k=0}^N \tilde{u}_{h,k}(t) \ T_k(x),\qquad h=1,2,3,
\end{equation}
where for every $k=0,\dots,N$, $T_k(x)=\cos\left(k\arccos(x)\right)$ are the Chebyshev polynomials of the first order, and $\tilde{u}_{h,k}(t)$ are the Chebyshev coefficients given by
\begin{equation}
\label{eq:Chebycoef}
\tilde{u}_{h,k} (t) = \frac{1}{\gamma_k}\sum_{n=0}^N u_h(x_n,t) T_k(x_n) w_n,\qquad h=1,2,3,
\end{equation}
where
\begin{equation}
\label{eq:gamma}
\gamma_k=\begin{cases}
    \pi&\qquad k=0,\,\text{or}\,k=N\\
    \frac{\pi}{2}&\qquad k=1,\dots,N-1
\end{cases}
\end{equation}
and
\begin{equation}
\label{eq:w}
w_n=\begin{cases}
    \frac{\pi}{2N}&\qquad k=0,\,\text{or}\,k=N\\
    \frac{\pi}{N}&\qquad k=1,\dots,N-1
\end{cases}
\end{equation}
Chebyshev polynomials are orthogonal with respect to the weight function $w(x)=1/sqrt{1-x^2}$, moreover, thanks to their definition, Chebyshev polynomials can be seen, under a suitable change of variables, as trigonometric cosine functions. Therefore, we can compute them efficiently by using the Fast Fourier Transform algorithm (FFT).

In order to deduce a semidiscrete formulation of the model, we have to replace $u_h(x,t)$ in~\eqref{eq:model} by its Chebyshev approximation $u_h^N(x,t)$, for every $h=1,2,3$.
\begin{equation}\label{eq:50}\frac{\partial}{\partial t} u_h^N(x,t) + \frac{\partial}{\partial x} I_N\left(f_h\left(u_h^N(x,t)\right)\right)=0,\end{equation}
where $I_N$ is the unique interpolation operator such that $I_N u_h^N(x_j,t)=u_h(x_j,t)$, for $j=0,\dots,N$ and $h=1,2,3$.

We want to study~\eqref{eq:50} in the frequency space. To do this, we need to recall for the relationship between Chebyshev coefficients of a function and Chebyshev coefficients of its partial derivative of any order. It is easy to verify that the Chebyshev coefficients of $\frac{\partial\ u_h^N}{\partial x}(x,t)$ are given by
\[\tilde{u}^{\,'}_{h,k}=\sum_{n=0}^N D_{kn}\tilde{u}_{h,n},\]
where $D$ is the derivative matrix, whose entries are given by
\[
D_{k,n}=\begin{cases}
n-1&\qquad k=1\,\text{and}\,n\,\text{even}\\
2(n-1)&\qquad k>1\,\text{and}\,n>k\,\text{and} n+k\,\text{odd}\\
0&\qquad \text{otherwise}
\end{cases}
\]
Moreover, we can relate the coefficients of a product of two functions to those of these two functions.
\begin{lemma}
\label{lm:product}
Let $\tilde{f}_k$ and $\tilde{g}_k$ be the Chebyshev coefficients of two functions $f$ and $g$, respectively, for $k=0,\dots,N$, then the $i$-th Chebyshev coefficient of the product function $fg$ is given by
\[
\tilde{\left(fg\right)}_i=\sum_{n=0}^N\sum_{m=0}^N P_{inm} \tilde{f}_n\tilde{g}_m
\]
where $P$ is a $(N+1)\times(N+1)\times(N+1)$ tensor, whose $N+1$ sub-matrix are defined as follows
\[
P_{1nm}=\begin{cases}
    1&\quad n=m=1\\
    \frac{1}{2}&\quad n>1\quad\text{and}\quad n=m\\
    0&\quad \text{otherwise}
\end{cases}
\]
\[
P_{2nm}=\begin{cases}
    1,&\quad (n=1,\,m=2)\quad\text{or}\quad(n=2,\,m=1)\\
    \frac{1}{2},&\quad |n-m|=1\quad\text{and}\quad n>2\\
    0,&\quad \text{otherwise}
\end{cases}
\]
and
\[
P_{inm}=\begin{cases}
    1,&\quad (n=i,\,m=1)\quad\text{or}\quad (n=1,\,m=i)\\
    \frac{1}{2}&\quad n+m=i+1\quad \text{and}\quad n=2,\dots,i-1\\
    \frac{1}{2},&\quad |n-m|=i-1\quad\text{and}\quad n=i+1,\dots,N+1
\end{cases}\quad i=3,\dots,N+1
\]
\end{lemma}

\begin{proof}
The proof is a simple application of the following recursive relationship among Chebyshev polynomials
\[
T_n(x)\ T_m(x) = \frac{1}{2}\left(T_{n+m}(x) + T_{|n-m|(x)}\right).
\]
\end{proof}

We replace $u_h(x,t)$ by $u_h^N(x,t)$ in the model~\eqref{eq:model}, for every $h=1,2,3$. Thus, if we make use of the previous results, at each collocation point, we can reformulate the model~\eqref{eq:50} in the frequency space as follows
\begin{equation}
\label{eq:discmethodinterior}
\frac{d}{dt}\tilde{u}_{h,k}(t) + \sum_{n=0}^N\sum_{m=0}^N\sum_{i=0}^N P_{knm}D_{mi}\tilde{f'}_{h,n}(t)\ \tilde{u}_{h,i}(t)=0, \qquad k=0,\dots,N,\quad h=1,2,3.
\end{equation}
where $\tilde{f'}_{h,n}$ is the $n$-th Chebyshev coefficients of $f'_h (u_h^N(x,t))$.

Using the same argument as before, we can approximate the initial conditions in the frequency space as follows
\begin{equation}
\label{eq:discmethodinitcond}
\tilde{u}_{h,k}(0) = \tilde{u_0}_{h,k},\qquad k=0,\dots N,\quad h=1,2,3.
\end{equation}

To close the discrete model, we need to compute and approximate the boundary conditions~\eqref{eq:BCcriterion} in the frequency space.

For every $t>0$ we have to find a zero of a scalar nonlinear function. It is easy to prove that the nonlinear function in~\eqref{eq:BCcriterion} is monotone and continuous but not everywhere differentiable since so is the Godunov flux. As a consequence, $u_b(t)$ can be computed by using Regula Falsi method.

Once the boundary value has been found, 
we impose
\begin{equation}
\label{eq:discmethodboundary}
\sum_{k=0}^N \tilde{u}_{1,k}(t)=\sum_{k=0}^N \tilde{u}_{2,k}(t)=\sum_{k=0}^N (-1)^k\  \tilde{u}_{3,k}(t)=u_b(t),\qquad t>0,
\end{equation}
so that, in the physical space we have
\begin{equation}
\label{eq:bcphysical}
u_1\left(x_N,t\right)=u_2\left(x_N,t\right)=u_3\left(x_0,t\right)=u_b(t),\qquad t>0.
\end{equation}

\subsection{Filtering for Chebyshev method}

Due to their high-order accuracy, spectral methods, when applied to nonlinear problem, may cause artificial numerical oscillations, known as Gibbs phenomenon. Such phenomenon can reduce drastically the accuracy of the method.

The formulation and the analysis of high-order accurate methods able to do not generate Gibbs phenomenon is an interesting challenge. Several results has been obtained in order to better understand the problem.

Filtering is a strategy to limit oscillations in the solution~\cite{Hesthaven}. It has been introduced in the framework of signal process with the aim to reduce the noises in a process. In our context, we implement a filter to make a specific modification of the Chebyshev coefficients in order to have a control of the oscillations in the profile of the solution.

In particular, a low-order filter dampens high-frequency components, while a high-order filters are able to localize oscillations close to discontinuities.

\begin{defn}
\label{def:filter}
A filter of order $p>0$ is a real function $\sigma\in\mathcal{C}^{p-1}(\R)$ satisfying the following properties
\begin{itemize}
    \item $\sigma(\eta) =0$ for $\eta<0$ and $\eta>1$,
    \item $\sigma(0)=1$,
    \item $\sigma(1)=0$,
    \item $\sigma^{(m)}(0)=\sigma^{(m)}(1)=0$ for $1\le m \le p-1$.
\end{itemize}
\end{defn}

In what follows we consider an exponential filter function defined as
\begin{equation}
 \label{eq:sigma}
 \sigma(\eta,p)=\begin{cases}
     \exp\left(-\beta\ \eta^p\right),&\quad 0\le\eta\le1\\
     0,&\quad \text{otherwise}
 \end{cases}
\end{equation}
where $p$ refers to the order of the filter, and $\beta$ is the damping rate. To ensure that the exponential filter~\eqref{eq:sigma} satisfies Definition~\ref{def:filter}, we choose $\beta=-\log(\eps_M)$, where $\eps_M$ is the machine accuracy. In Figure~\ref{fig:filtro} we show the behavior of exponential filters of different orders.

\begin{figure}
    \centering
    \includegraphics[width=0.5\textwidth]{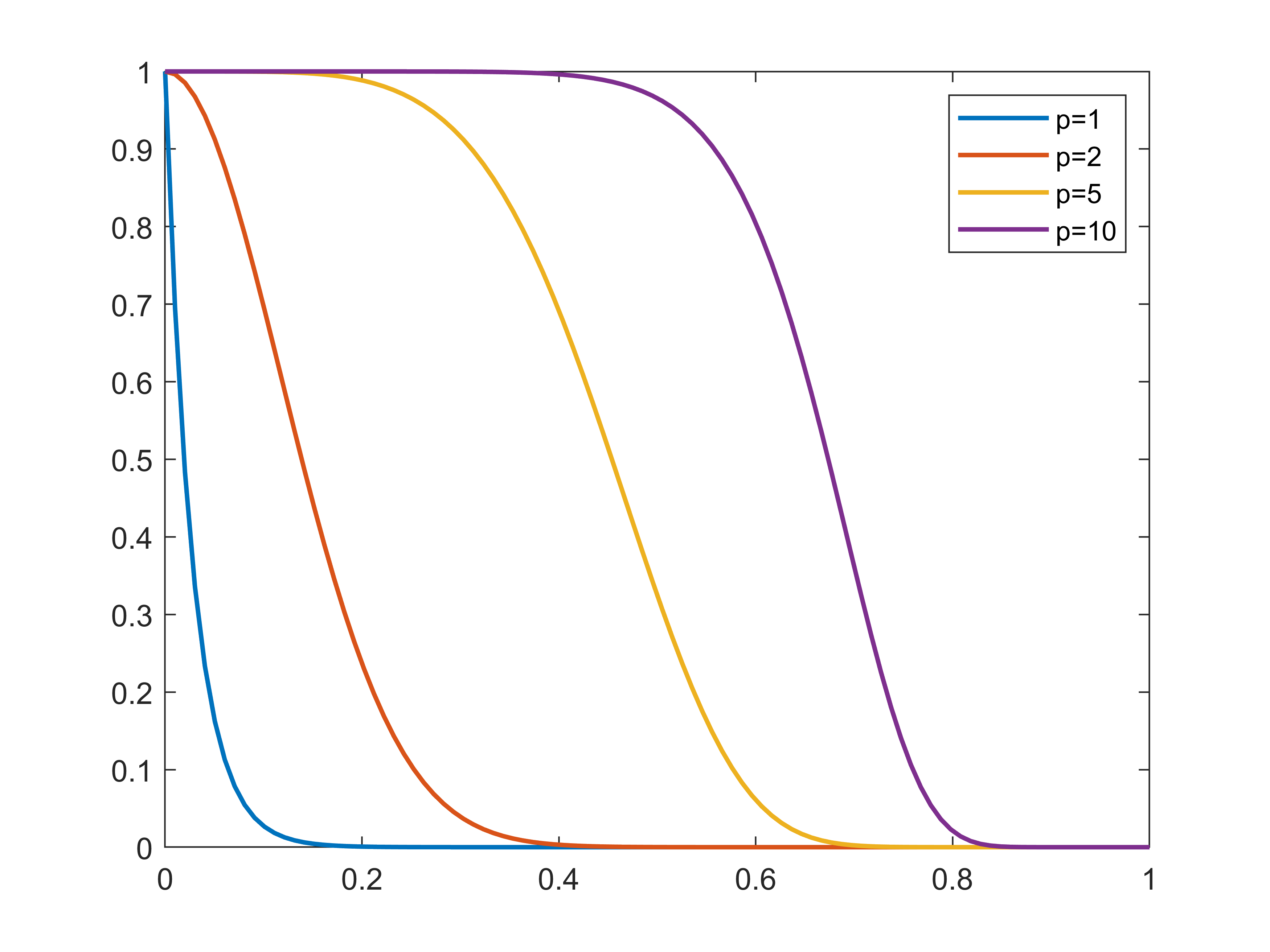}
    \caption{Different exponential filters at varying of the order $p$.}
    \label{fig:filtro}
\end{figure}

To apply the filtering strategy to our model, we modify the $k$-th Chebyshev coefficient defined in~\eqref{eq:Chebycoef} by pre-multiplication with $\sigma\left(\frac{k}{N}\right)$, namely the $k$-th filtered coefficients becomes
\begin{equation}
\label{eq:FChebycoef}
\tilde{u}_{h,k}^{\,\,\sigma}(t)=\sigma\left(\frac{k}{N}\right)\ \tilde{u}_{h,k}(t).
\end{equation}
Thus, the semi-discrete method~\eqref{eq:discmethodinterior} becomes
\begin{equation}
\label{eq:filterdiscmethodinterior}
\frac{d}{dt}\tilde{u}_{h,k}^{\,\,\sigma} (t) + \sum_{n=0}^N\sum_{m=0}^N\sum_{i=0}^N P_{knm}D_{mi}\tilde{f'}_{h,n}^{\,\,\sigma} (t)\ \tilde{u}_{h,i}^{\,\,\sigma}(t)=0, \qquad k=0,\dots,N,\qquad h=1,2,3.
\end{equation}


\subsection{Super Spectral Viscosity}

When spectral methods are applied to nonlinear hyperbolic equations in conservation form, the problem of obtaining an entropy satisfying solution arises. Unmodified spectral methods do not converge to the correct entropy solution if the solution contains shocks~\cite{Tadmor1989}.

Spectral calculations can be stabilized by using exponential filters on the conserved variable at each time step. An equivalent approach consists in applying a Super Spectral Viscosity (SSV) in the discretized model. Indeed, a SSV method converges to the correct entropy solution and maintains the spectral accuracy, when the solution is uniformly bounded.

In this section we introduce the Super Spectral Viscosity in the semi-discrete scheme and we show the equivalence of such approach with the implementation of an exponential filter. As we will see in the next section, the introduction of the SSV method allows us to prove the convergence of the semi-discrete scheme.

We defined the Super Spectral Viscosity for the Chebyshev collocation method as follows
\begin{equation}
    \label{eq:SSV}
    SV(u_h^N) = \frac{\eps(-1)^{s+1}}{N^{2s-1}}\left(\sqrt{1-x^2}\frac{\partial}{\partial x}\right)^{2s} \left(u_h^N\right) = \frac{\eps(-1)^{s+1}}{N^{2s-1}}Q^{2s} \left(u_h^N\right), \quad h=1,2,3,
\end{equation}
where $\eps>0$ is the super viscosity coefficient, and $s$ is an integer growing with $N$.

Thus, the Chebyshev collocation method can be written in the following way
\begin{equation}
    \label{eq:ssvmethod}
    \frac{\partial}{\partial t} u_h^N(x,t) + \frac{\partial}{\partial x} I_N\left(f_h\left(u_h^N\right)\right) = \frac{\eps(-1)^{s+1}}{N^{2s-1}}\left(\sqrt{1-x^2}\frac{\partial}{\partial x}\right)^{2s} \left(u_h^N\right),\quad h=1,2,3.
\end{equation}

\begin{prop}
Applying the SSV method to the Chebyshev collocation method is equivalent to apply the exponential filter~\eqref{eq:sigma} with $\beta=\eps N \Delta t$ and $p=2s$.
\end{prop}

\begin{proof}
   Let start by examining the super spectral viscosity operator applied to the Chebyshev polynomial $T_k(x)$ for $s=1$:
   \begin{equation*}
       \begin{split}
    Q^2\left(T_k(x)\right) &= \sqrt{1-x^2} \frac{\partial}{\partial x} \left(\sqrt{1-x^2} \frac{\partial}{\partial x} T_k(x)\right)\\
           &=-k^2 \ T_k(x)
       \end{split}
   \end{equation*}

   This means that Chebyshev polynomials are the eigenfunctions of the operator $Q^2$ with eigenvalues $-k^2$.

   We have
   \begin{equation*}
     \begin{split}
         \frac{\eps (-1)^{s+1}}{N^{2s-1}} Q^{2s} u_h^N &= \frac{\eps (-1)^{s+1}}{N^{2s-1}} Q^{2s} \sum_{k=0}^N \tilde{u}_{h,k}(t) \ T_k(x)\\
         &=-\eps N \sum_{k=0}^N \left(\frac{k}{N}\right)^{2s} \tilde{u}_{h,k}(t) \ T_k(x).
     \end{split}  
   \end{equation*}

   We implement the SSV method via time splitting, where in the first step we solve
   \begin{equation}
    \label{eq:4}
    \frac{\partial}{\partial t} u_h^N + \frac{\partial}{\partial x} I_N\left(f_h\left(u_h^N\right)\right) =0
   \end{equation}
   and in the second step we solve
   \begin{equation}
       \label{eq:5}
       \frac{\partial}{\partial t} u_h^N = \frac{\eps(-1)^{s+1}}{N^{2s-1}} Q^{2s} u_h^N.
   \end{equation}
   Then, in the frequency space, the seocnd equation~\eqref{eq:5} can be written as
   \begin{equation*}
   \frac{d}{dt} \tilde{u}_{h,k}(t) = -\eps N \left(\frac{k}{N}\right)^{2s} \tilde{u}_{h,k}(t),
\end{equation*}
whose exact solution over one time step is
\[
\tilde{u}_{h,k}(t+\Delta t) = \tilde{u}_{h,k}(t) \exp\left(-\eps N \Delta t \left(\frac{k}{N}\right)^{2s}\right)
\]
Thus, the exact solution of the SSV split step is a filtered partial sum
\[
u_h^N(x,t)=\sum_{k=0}^N \sigma\left(\frac{k}{N}\right)\tilde{u}_{h,k}(t) T_k(x)
\]
with $\beta=\eps N \Delta t$ and $p=2s$.
\end{proof}

The reformulation of the filtered Chebyshev method by using Super Spectral Viscosity, allows us to prove the convergence of the semi-discrete scheme.

\subsection{Convergence of the semi-discrete scheme}

We consider the spectral approximation~\eqref{eq:ssvmethod} of~\eqref{eq:model}, with initial condition $u_{0,h}^N$, $h=1,2,3$. For simplicity, we focus on the case $s=1$, but the results of this section can be easily generalized to $s>1$.

Due to the Chebyshev approximation properties, we can assume
\begin{equation}
    \label{eq:assump1}
    \begin{split}
    &u_{0,h}^N \to u_{0,h}\quad\text{in}\quad L^p_{loc}(\Omega_h),\quad 1\le p <\infty,\quad\text{as}\quad N\to\infty,\\
    &\norm{u_{0,h}^N}_{L^\infty\left(\Omega_h\right)}\le \norm{u_{0,h}}_{L^\infty\left(\Omega_h\right)},\quad \norm{u_{0,h}^N}_{L^2\left(\Omega_h\right)}\le \norm{u_{0,h}}_{L^2\left(\Omega_h\right)},\quad N>0.
    \end{split}
\end{equation}

Additionally, the maximum principle for parabolic equations applied to~\eqref{eq:ssvmethod} ensures
\begin{equation}
    \label{eq:assump2}
    \norm{u_{h}^N}_{L^\infty\left(\Omega_h\times (0,\infty)\right)} \le \norm{u_{0,h}}_{L^\infty\left(\Omega_h\right)},\quad N>0,\quad h=1,2,3.
\end{equation}

We prove the following preliminary Lemmas.
\begin{lemma}
 Let $u_h^N$ be the solution of~\eqref{eq:ssvmethod}, for $h=1,2,3$, then, the following energy estimate holds
 \begin{equation}
     \label{eq:energyest}   \norm{u_h^N(\cdot,t)}_{L^2(\Omega)}^2 +  \frac{2\eps}{N} \int_0^t \norm{\sqrt{1-x^2}\frac{\partial}{\partial x} u_N(\cdot,s)}_{L^2(\Omega)}^2\,ds \le \norm{u_0}_{L^2(\Omega)}^2,\qquad t\ge0,\,N>0.
 \end{equation}
\end{lemma}

\begin{proof}
    We have that
    \begin{align*}
        \frac{d}{dt} \int_{-1}^1\frac{\left(u^N_h\right)^2}{2}dx &=\int_{-1}^1 u_h^N\frac{\partial }{\partial t} u_h^N dx\\
        &=\frac{\eps}{N}\int_{-1}^1 u_h^N \left(\sqrt{1-x^2}\frac{\partial}{\partial x}\right)^2 u_h^N dx - \int_{-1}^1 u_h^N f'_h\left(u_h^N\right)\frac{\partial }{\partial x} u_h^N\\
        &=-\frac{\eps}{N}\int_{-1}^1 \left(\sqrt{1-x^2}\right)^2\left(\frac{\partial}{\partial x} u_h^N\right)^2 dx \\
				&\quad- \underbrace{\int_{-1}^1 \frac{\partial}{\partial x}\left(\sqrt{1-x^2} \int_0^{u_h^N(x,t)} s f'_h(s) ds\right) dx}_{{}=0}\\
        &=-\frac{\eps}{N}\int_{-1}^1 \left(\sqrt{1-x^2}\right)^2\left(\frac{\partial}{\partial x} u_h^N\right)^2 dx,
    \end{align*}
    that is
    \begin{equation}
    \label{eq:1}
    \frac{d}{dt} \norm{u^N_h(\cdot,t)}_{L^2(\Omega_h)}^2+\frac{2\eps}{N}\norm{\sqrt{1-x^2}\frac{\partial }{\partial x }u_h^N(\cdot,t)}_{L^2(\Omega_h)}^2 =0.
    \end{equation}
    Integrating over $(0,t)$ equation~\eqref{eq:1} and using~\eqref{eq:assump1}, we get the claim.
\end{proof}

We can now prove the convergence of the semi-discrete scheme.

\begin{theorem}
\label{th:conv}
Consider the spectral approximation~\eqref{eq:ssvmethod} of~\eqref{eq:model}. 
If its solution $u_h^N$ is uniformly bounded in $L^\infty(\Omega_h\times[0,+\infty])$, for $h=1,2,3$, then there exists a subsequence $\left\{u_h^{N_m}\right\}_m$ and a function $u_h\in L^\infty(\Omega_h\times(0,+\infty))$ such that
\begin{equation}
\label{eq:convLp}
u_h^{N_m}\to u_h \qquad \text{in}\quad L^p_{loc}(\Omega_h),\quad 1\le p<\infty\quad\text{and a.e. in}\quad\Omega_h.
\end{equation}
\end{theorem}

\begin{proof}

We want to prove the result by using the Compensated Compactness Theorem~\cite{Tartar}.

We fix $T>0$ and the entropy $\eta_h\in\mathcal{C}^2(\R)$ 
with flux $q_h(u_h^N,c)$
such that $q'_h=f'_h\eta_h'$, for $h=1,2,3$.

We want to prove that 
\begin{equation}
\label{eq:etacompact}
\left\{\frac{\partial}{\partial t} \eta_h\left(u_h^N\right) + \frac{\partial }{\partial x} q_h\left(u_h^N\right)\right\}_N \quad \text{is compact in}\quad H^{-1}\left(\R\times (0,T)\right),\qquad\text{for}\quad h=1,2,3.
\end{equation}

We observe that
\begin{equation}
\label{eq:2}
\begin{split}
    \frac{\partial}{\partial t} \eta_h\left(u_h^N\right) + \frac{\partial }{\partial x} q_h\left(u_h^N\right)&=\eta'_h\left(u_h^N\right)\left(\frac{\partial }{\partial t}u_h^N + \frac{\partial}{\partial x} f_h\left(u_h^N\right)\right)\\
    &=\frac{\eps}{N}\left(\sqrt{1-x^2}\frac{\partial}{\partial x}\right)^2 \left(u_h^N\right)\\
    &=\frac{\eps}{N} \left(\sqrt{1-x^2}\right)^2 \frac{\partial^2}{\partial x^2} \left(\eta_h\left(u_h^N\right)\right)\\
		&\quad - \frac{\eps}{N} \left(\sqrt{1-x^2}\right)^2\eta''_h\left(u_h^N\right)\left(\frac{\partial}{\partial x} u_h^N\right)^2.
\end{split}
\end{equation}

Since $\sqrt{1-x^2}\le 1$, using~\eqref{eq:energyest} and 
\[
\frac{\eps}{N}\frac{\partial^2}{\partial x^2} \left(\eta_h\left(u_h^N\right)\right) = \frac{\partial}{\partial x} \left(\frac{\eps}{N} \eta'_h\left(u_h^N\right)\frac{\partial}{\partial x}u_h^N\right),
\]
we obtain
\begin{equation*}
    \begin{split}
        \norm{\frac{\eps}{N}\eta'_h\left(u_h^N\right)\frac{\partial}{\partial x}u_h^N}^2_{L^2\left(\Omega_h\times (0,T)\right)}&\le \frac{\eps^2}{N^2}\norm{\eta'_h}^2_{L^\infty\left(-\norm{u_{0,h}}_{L^\infty(\Omega_h)},\norm{u_0}_{L^\infty(\Omega_h)}\right)} \int_0^T \norm{\frac{\partial}{\partial x} u_h^N (\cdot,s)}_{L^2(\Omega_h)} ds\\
        &\le \frac{\eps}{N} \norm{\eta'_h}^2_{L^\infty\left(-\norm{u_{0,h}}_{L^\infty(\Omega_h)},\norm{u_0}_{L^\infty(\Omega_h)}\right)} \norm{u_{h,0}}_{L^2(\Omega_h)}^2 \to 0,
    \end{split}
\end{equation*}
therefore,
\begin{equation}
    \label{eq:firsttermcompact}
    \left\{\frac{\eps}{N} \left(\sqrt{1-x^2}\right)^2 \frac{\partial^2}{\partial x^2} \left(\eta_h\left(u_h^N\right)\right)\right\}_N\quad\text{is compact in}\quad H^{-1}\left(\Omega_h\times(0,T)\right).
\end{equation}

Moreover,
\begin{align*}
     &   \norm{\frac{\eps}{N} \left(\sqrt{1-x^2}\right)^2\eta''_h\left(u_h^N\right)\left(\frac{\partial}{\partial x} u_h^N\right)^2}_{L^1(\Omega_h\times (0,T))}\\
			&	\qquad \le \frac{\eps}{N}\norm{\eta''_h}_{L^\infty\left(-\norm{u_{0,h}}_{L^\infty(\Omega_h)},\norm{u_0}_{L^\infty(\Omega_h)}\right)} \int_0^T \norm{\frac{\partial}{\partial x} u_h^N (\cdot,s)}_{L^2(\Omega_h)}^2 ds\\
       &\qquad +\frac{1}{2} \norm{\eta''_h}_{L^\infty\left(-\norm{u_{0,h}}_{L^\infty(\Omega_h)},\norm{u_0}_{L^\infty(\Omega_h)}\right)} \norm{u_{h,0}}_{L^2(\Omega_h)}^2, 
\end{align*}
thus 
\begin{equation}
    \label{eq:secondtermcompact}
    \left\{\frac{\eps}{N} \left(\sqrt{1-x^2}\right)^2\eta''_h\left(u_h^N\right)\left(\frac{\partial}{\partial x} u_h^N\right)^2\right\}_N\quad\text{is compact in}\quad\mathcal{M}\left(\Omega_h\times (0,T)\right),
\end{equation}
where $\mathcal{M}\left(\Omega_h\times (0,T)\right)$ denotes the space of measures on $\Omega_h\times (0,T)$.

Thanks to~\eqref{eq:firsttermcompact} and~\eqref{eq:secondtermcompact}, we can use Murat's Lemma and find that
\begin{equation}
    \label{eq:compactnesseta}
    \left\{\frac{\eps}{N} \left(\sqrt{1-x^2}\right)^2 \frac{\partial^2}{\partial x^2} \left(\eta_h\left(u_h^N\right)\right) - \frac{\eps}{N} \left(\sqrt{1-x^2}\right)^2\eta''_h\left(u_h^N\right)\left(\frac{\partial}{\partial x} u_h^N\right)^2\right\}\,
\end{equation}
is compact in $H^{-1}\left(\Omega_h\times (0,T)\right)$, providing~\eqref{eq:etacompact}.

Hence, we are able to apply the Compensated Compactness Theorem, finding that there exists a subsequence $\{u_h^{N_m}\}_m$ and a function $u_h\in L^\infty\left(\Omega_h\times (0,\infty)\right)$ such that
\begin{equation}
    \label{eq:thesis}
    u_h^{N_m} \to u_h\quad \text{in}\quad L^p_{loc}\left(\Omega_h\times (0,\infty)\right),\quad 1\le p<\infty,\quad\text{and a.e. in}\quad \Omega_h. 
\end{equation}

Now, we have to show that $u$ is an entropy weak solution to~\eqref{eq:model}. Let $\varphi\in\mathcal{C}^{\infty}(\R^2)$ be a positive test function with compact support.

We have to prove that
\begin{equation}
    \label{eq:entropythesis}
    \int_0^\infty \int_{\Omega_h} \left(\eta_h\left(u_h\right)\frac{\partial}{\partial t}\varphi +q_h\left(u_h\right)\frac{\partial}{\partial x}\varphi\right) dtdx + \int_{\Omega_h} \eta_h\left(u_0(x)\right) \varphi(x,0) dx \ge 0.
\end{equation}

From~\eqref{eq:2}, we have
\[
\frac{\partial}{\partial t}\eta_h\left(u_h^{N_m}\right) + \frac{\partial}{\partial x} q_h\left(u_h^{N_m}\right) \le \frac{\eps}{N_m} \frac{\partial^2}{\partial x^2} \left(\eta_h\left(u_h^{N_m}\right)\right).
\]
Multiplying by $\varphi$ and integrating over $\Omega_h\times (0,\infty)$ we have that
\begin{equation}
    \label{eq:3}
    \begin{split}
        \int_0^\infty \int_{\Omega_h} \left(\eta_h\left(u_h^{N_m}\right)\frac{\partial}{\partial t}\varphi +q_h\left(u_h^{N_m}\right)\frac{\partial}{\partial x}\varphi\right) dtdx &+ \int_{\Omega_h} \eta_h\left(u_0^{N_m}(x)\right) \varphi(x,0) dx\\
        &+ \frac{\eps}{N_m} \int_0^{\infty}\int_{\Omega_h} \eta_h\left(u_h^{N_m}\right)\frac{\partial^2}{\partial x^2} \varphi dtdx \ge 0.
    \end{split}
\end{equation}
Thus, equation~\eqref{eq:entropythesis} follows from~\eqref{eq:assump1}, \eqref{eq:assump2}, \eqref{eq:3} and the dominated convergence theorem.
\end{proof}


\section{Fully-discrete scheme}
\label{sec:fully}

Once the Cauchy problem on network~\eqref{eq:model} has been discretized in space, we need to integrate in time the obtained system of ODEs~\eqref{eq:discmethodinterior}. In this paper, we construct the fully-discrete scheme by implementing the  Midpoint method.

We fix the time step $\Delta t>0$ satisfying the CFL condition
\begin{equation}
    \label{eq:CFL}
    \Delta t \max_h \left(L_h\right)\le \frac{\min_{j=0,\dots,N-1} \left|x_{j+1}-x_j\right|}{2},
\end{equation}
and for $s\in\N$ we define the time mesh $t^s=s\Delta t$. At each time step $t^s$, $\tilde{u}_{h,k}^{\, s}$ represents an approximation of the solution $\tilde{u}_{h,k}(t^s)$, for every $k=0,\dots,N$ and $h=1,2,3$.

Let us introduce the following notation for the second term of the left hand side of~\eqref{eq:filterdiscmethodinterior}
\begin{equation}
\label{eq:discflux}
F_{h,k}^\sigma\left(t^s,\tilde{u}_h^\sigma(t^s)\right) = \sum_{n=0}^N\sum_{m=0}^N\sum_{i=0}^N P_{knm}D_{mi}\tilde{f'}_{h,n}^{\,\sigma}(t^s)\ \tilde{u}_{h,i}^{\,\sigma}(t^s)=0, \qquad k=0,\dots,N,\qquad h=1,2,3.
\end{equation}

The fully-filtered discrete scheme is given by
\begin{equation}
\label{eq:fullymethod}
\tilde{u}_{h,k}^{\,s+1} = \tilde{u}_{h,k}^{\,s} - \Delta t F_{h,k} \left(t^s+\frac{\Delta t}{2}, \tilde{u}_{h} \left(t^s+\frac{\Delta t}{2}\right)\right),\qquad k=0,\dots, N,\qquad h=1,2,3.
\end{equation}


\subsection{A 2D Chebyshev collocation method for the fully discretized model}
\label{sec:fully-2D}

In order to obtain a second order scheme in time, one can use a spectral Chebyshev method both for space and time discretization, as shown in~\cite{LPcheby2022}.

The method consists in looking for an approximation of the solution on the whole network in the form of finite combination of the product of Chebyshev polynomials in space and time variables.

More precisely, for simplify the notation, we assume that the time interval under consideration is $[-1,1]$, so that the initial condition corresponds to $t=-1$. We consider as full domain of computation the product $\tilde{\Omega}=\Omega\times[-1,1]$ and we look for $u^N_h(x,t)\approx u_h(x,t)$, $h=1,2,3$ given by
\begin{equation}
\label{eq:chebser2d}
u^N_h(x,t)=\sum_{j=0}^{N_x}\sum_{k=0}^{N_t} \tilde{u}_{h,jk}\ T_j(x)\ T_k(t)
\end{equation}
where $N_x$ and $N_t$ represent the total number of collocation points in space and time, respectively, and $N=(N_x,N_t)$.

The coefficients $\tilde{u}_{h,jk}$, $j=0,\dots,N_x$
, $k=0,\dots,N_t$, and $h=1,2,3$, are the Chebyshev coefficients of the discrete Chebyshev expansion and when the grid points are the Chebyshev Gauss-Lobatto points $(x_n,t_m)=(\cos(n\pi/N_x),\cos(m\pi/N_t))$, then their expression is given by
\begin{equation}
\label{eq:fcoef}
\tilde{u}_{h,jk}=\frac{1}{\gamma_j\gamma_k}\sum_{n=0}^{N_x}\sum_{m=0}^{N_t} u_h(x_n,t_m)\ T_j(x_n)T_k(t_m)w_n w_m,
\end{equation}
where $\gamma$ and $w$ are defined in~\eqref{eq:gamma} and~\eqref{eq:w}, respectively.

For the purpose of our work, we show here the expression of the expansion of the first order derivative in time and in space:
\begin{equation}
\label{eq:ptchebser}
\begin{split}
\frac{\partial }{\partial t}u^N_h(x,t)&=\sum_{j=0}^{N_x}\sum_{k=0}^{N_t}\sum_{\ell=0}^{N_t} D_{k\ell}^{(t)}\ \tilde{u}_{h,j\ell}\ T_j(x)\ T_k(t),\\
\frac{\partial }{\partial x}u^N_h(x,t)&=\sum_{j=0}^{N_x}\sum_{k=0}^{N_t}\sum_{i=0}^{N_x} D_{ji}^{(x)}\ \tilde{u}_{h,ik}\ T_j(x)\ T_k(t)
\end{split}
\end{equation}
where $D$ is the derivative matrix, 
and the superscripts $(t)$ and $(x)$ in such a matrix denote the differentiation with respect to the temporal or spatial coordinates, respectively.

Moreover, it is possible to find a compact expression for the product of two functions as linear combination of Chebyshev polynomials in the same way as in Lemma~\ref{lm:product}. We have that the $ij$-th Chebyshev coefficient of the product between two functions $f$ and $g$ is given by
\begin{equation}
\label{eq:cheb2dproduct}
\tilde{\left(fg\right)}_{ij}=trace\left(P^{(x)}_i\tilde{f}\ \tilde{g}^{\, T}\ P^{(t)}_j\right),\qquad i=0,\dots, N_x,\quad j=0,\dots,N_t+1,
\end{equation}
where $\tilde{f}$ and $\tilde{g}$ are the $(N_x+1)\times (N_t+1)$ matrix of Chebyshev coefficients of $f$ and $g$, while $P^{(x)}$ and $P^{(t)}$ are the $(N_x+1)\times(N_x+1)\times(N_x+1)$ and $(N_t+1)\times(N_t+1)\times(N_t+1)$ tensors of Lemma~\ref{lm:product}, respectively.

The advantage of implementing such scheme is related to the fact that even in this case we can benefit of the implementation of the 2D-FFT to compute the Chebyshev coefficients $\tilde{u}_{h,jk}$ in~\eqref{eq:chebser2d}.

If we replace $u_{h}(x,t)$ with its Chebyshev approximation $u_h^N(x,t)$, we find
\begin{align*}
    \label{eq:2Dapprox}
    0=&\frac{\partial}{\partial t} u_h^N(x,t) + \frac{\partial }{\partial x} I_N\left(f_h(u_h^N(x,t))\right)\\
    =&\sum_{j=0}^{N_x}\sum_{k=0}^{N_t}\left(\sum_{\ell=0}^{N_t} D_{k\ell}^{(t)}\ \tilde{u}_{h,j\ell} + a_{h,jk}\right)\ T_j(x)\ T_k(t),
\end{align*}
where for every $h=1,2,3$, $j=0,\dots,N_x$ and $k=0,\dots,N_t$, $a_{h,jk}$ is given by
\begin{equation}
    \label{eq:a}
    a_{h,jk}=trace\left(P_j^{(x)}\tilde{(f')}_h\left(\sum_{\ell =0}^{N_x} D^{(x)}_{j\ell}\tilde{u}_{h,\ell k}\right)^{\,T}\ P_k^{(t)}\right).
\end{equation}

The unknowns $\tilde{u}_{h,jk}$ have to satisfy the conservation law in the interior of the integration domain with special care for imposing the initial and boundary conditions.
We have
\begin{align*}
\label{eq:initcond2D}
u^N_{h}(x,-1)&=\sum_{j=0}^{N_x}\sum_{k=0}^{N_t} (-1)^k \tilde{u}_{h,jk} T_j(x)\\
&=\sum_{j=0}^{N_x} \tilde{u_0}_{h,j} T_j(x)\qquad\qquad\qquad \qquad\qquad \qquad \qquad\qquad h=1,2,3.\\
&=u_{h,0}(x)
\end{align*}
The boundary conditions for the incoming edges can be obtained by solving the system
\begin{equation}
\label{eq:BCincom2D}
\sum_{j=0}^{N_x} \tilde{u}_{h,jk} = \tilde{(u_b)}_k\qquad h=1,2,\quad  k=0,\dots,N_t.
\end{equation}
While the boundary condition for the outgoing edge writes
\begin{equation}
\label{eq:BCoutg2D}
\sum_{j=0}^{N_x} (-1)^k \ \tilde{u}_{h,jk} = \tilde{(u_b)}_k\qquad h=3,\quad  k=0,\dots,N_t.
\end{equation}

We can notice that, in order to find the solution on the whole network, we need to impose the boundary condition for every time step $t_k$, $k=0,\dots,N_t$. However, the proposed Riemann solver used to select a unique weak solution to~\eqref{eq:cl}, requires to compute the boundary values implicitly at each iteration. So, in our analytical framework, a 2D Chebyshev spectral method is not suitable to get the admissible weak solution.

Such scheme can be applied successfully to each edge, but it needs a further analytical investigation to be couple with a different transmission condition. This represents an ongoing study, which will be presented in a future work.

\section{Numerical Simulations}\label{sec:numericalSimulations}

In this section we present some simulations in order to show the properties of the proposed filtered Chebyshev spectral method applied to the network. 

More in details, we make a validation of the scheme by comparing the approximated solution provided by the method with an exact solution on network computed explicitly in~\cite{Pellegrino}. Then, we compare the filtered Chebyshev scheme with the finite volume scheme proposed in~\cite{Pellegrino} in terms of CPU cost.

Finally, in order to verify that, at least far from the junction, even 2D-Chebyshev method constructed in~\ref{sec:fully} performs well, in the framework of hyperbolic conservation laws.

\subsection{Validation of the scheme on network}
\label{sec:validation}

We fix $N>0$ and discretize $\Omega=[-1,1]^3$ by using the CGL points. The validation of the filtered Chebyshev method on network is made by comparing the obtained approximated solution with the explicit solution computed in~\cite{Pellegrino}, with flux $f_h(u) = u\left(1-u\right)$, for every $h=1,2,3$, and with initial condition
\[
u_{1,0}(x) = \chi_{[-1/2,0]}(x),\quad u_{2,0}(x)=\frac{3}{4} \chi_{[-1/4,0]}(x),\quad u_{3,0}(x)=0.
\]

Figure~\ref{fig:compare} compare the profile of the numerical solution on each edge with the exact one, showing a good agreement and the capability of the spectral method to capture well the shock without showing the appearance of spurious oscillations near discontinuities. The result is also confirmed in terms of convergence rate.

We introduce the relative $L^1$-error on network as follows
\[
E^s_{L^1}=\frac{\sum_{h=1}^3 \sum_{j=0}^N \left|u^N_h(x_j,t^s)-u^\ast_h(x_j,t^s)\right|}{\sum_{h=1}^3\sum_{j=0}^N \left|u^\ast(x_j,t^s)\right|},
\]
where $u^\ast$ is the reference solution.


Table~\ref{tab:error} shows the relative error $E^{s}_{L^1}$ corresponding to the incoming edges of the network for different values of the total number of mesh points $N$ at time $t^s=1/6$ and the associated convergence rate. The table shows also a comparison between the error and the convergence rate obtained with the finite volume scheme (FVS) proposed in~\cite{Pellegrino}.

We can conclude that thanks to spectral method we can gain one order of accuracy, even in presence of shocks in the profile of the solution.

\begin{figure}
 \centering
 \subfloat[][Profile of the solution on the first incoming edge for $t=1/6$.]
 {\includegraphics[width=.3245\textwidth]{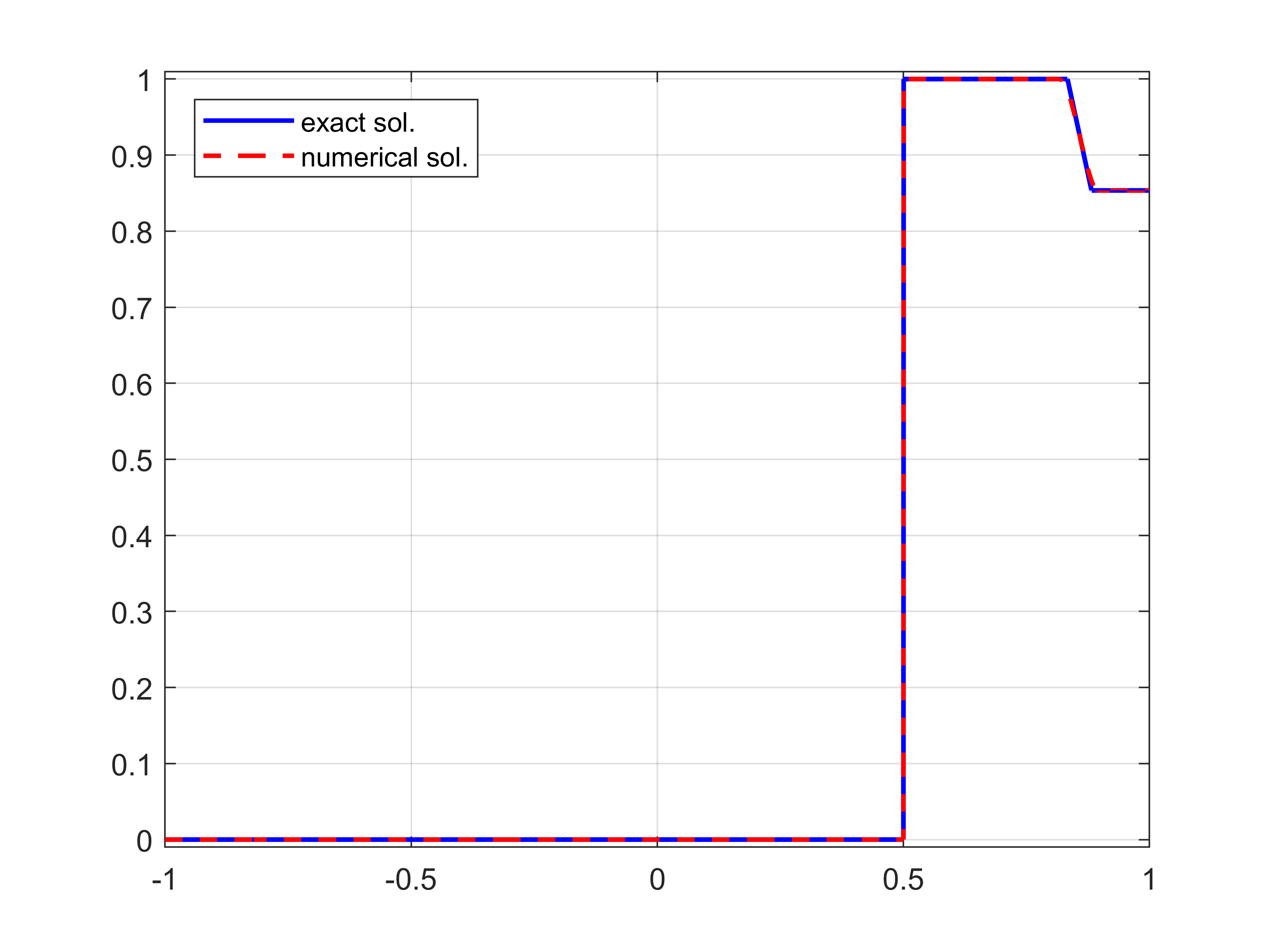}\label{fig:compare_delta0.15}}\,
 \subfloat[][Profile of the solution on the second incoming edge for $t=1/6$.]
 {\includegraphics[width=.3245\textwidth]{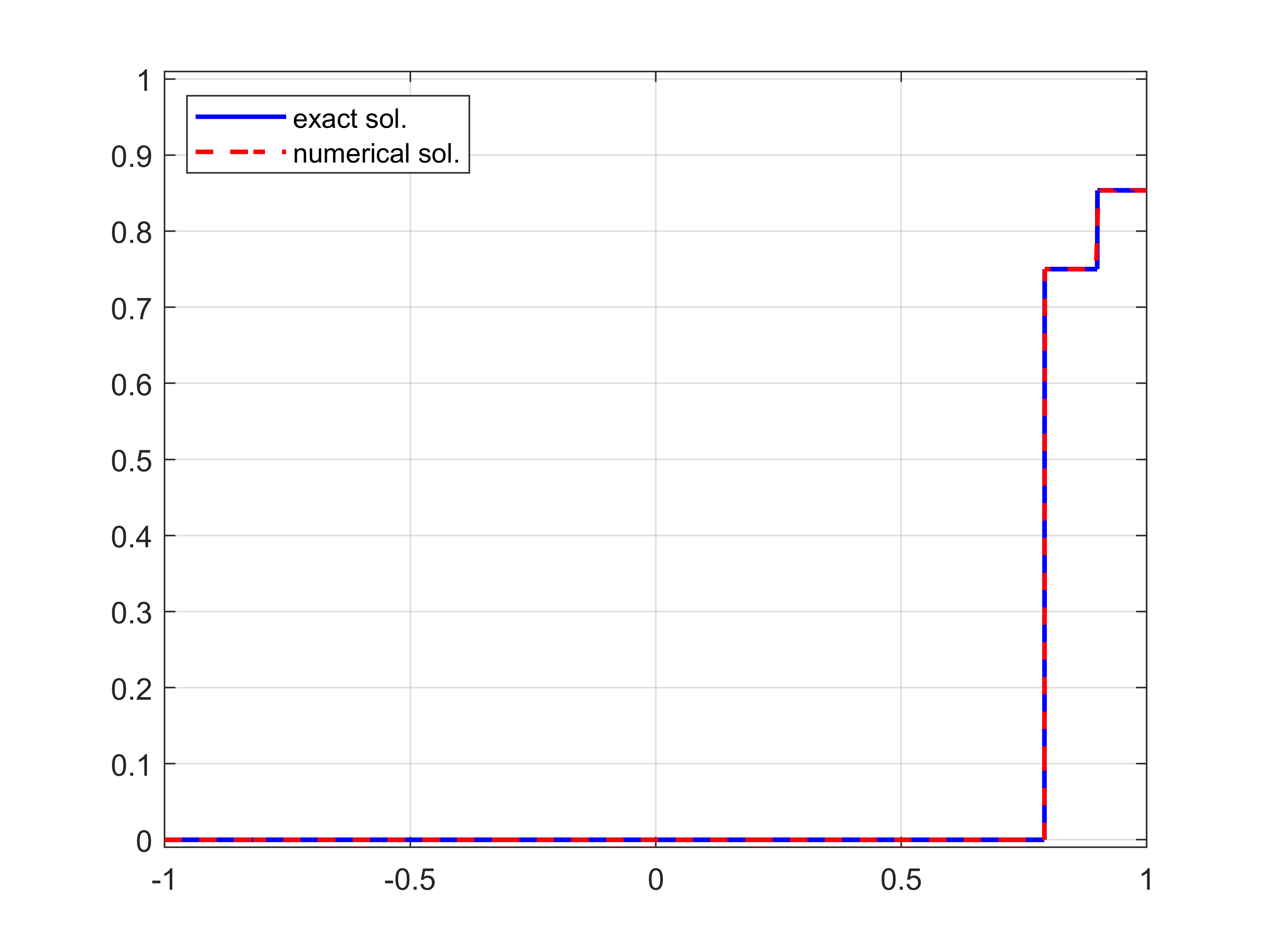}\label{fig:compare_delta0.3}} \,
 \subfloat[][Profile of the solution on the outgoing edge for $t=1/6$.]
 {\includegraphics[width=.3245\textwidth]{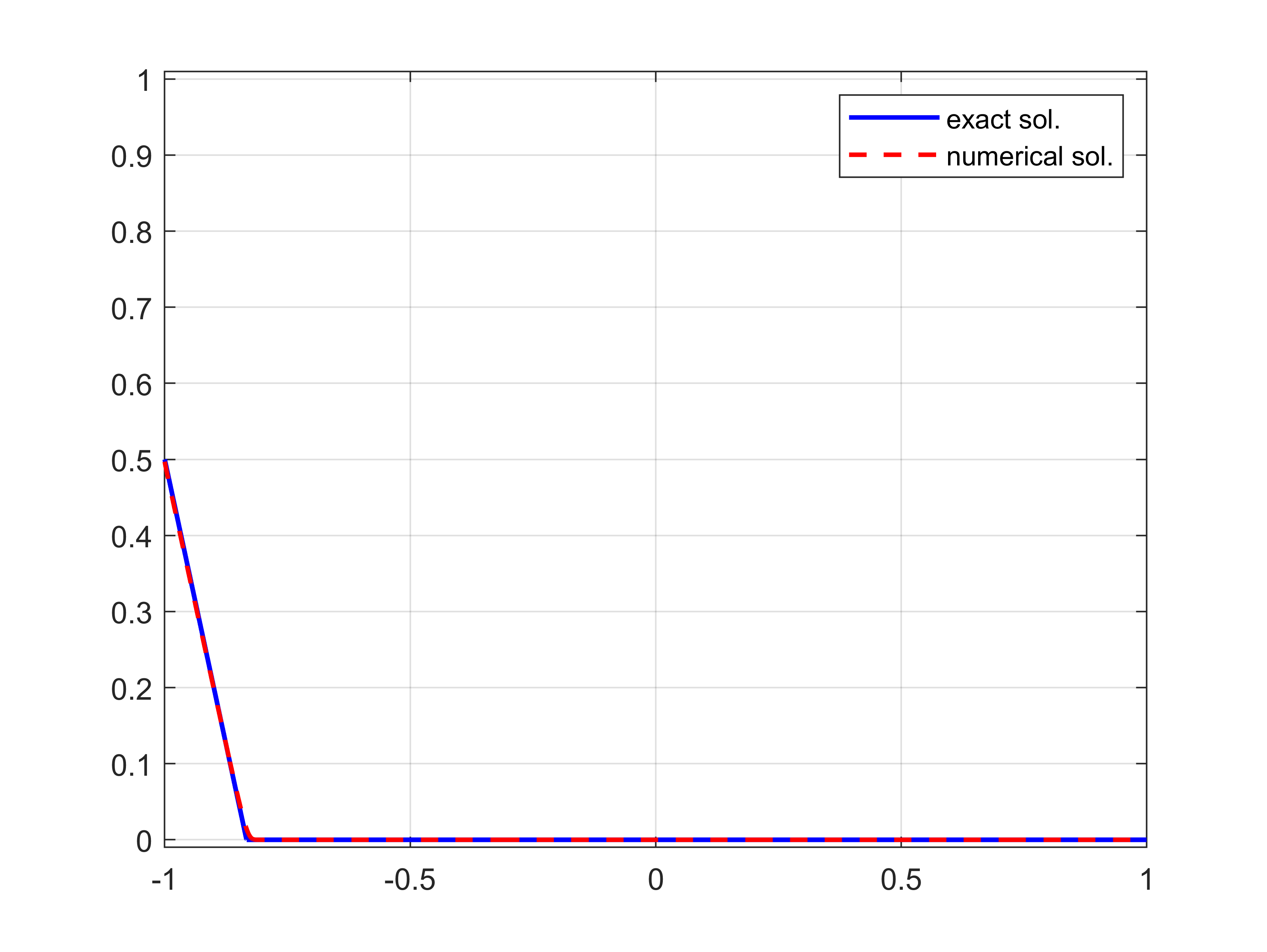}\label{fig:compare_delta0.5}}
 \caption{With reference to the simulation of Section~\ref{sec:validation}: comparison between the numerical solution obtained by implementing the filtered Chebyshev scheme with the explicit solution at a fixed time. The parameters for the simulation are $N= 4000$, $\Delta t=10^{-4}$, $p=10$. 
 }
  \label{fig:compare}
 \end{figure}

\begin{table}%
\centering%
\renewcommand\arraystretch{1.3}
\begin{tabular}{ccccc
}
\toprule
\multirow{2}*{$N$} & \multicolumn{2}{c}{FVS } & \multicolumn{2}{c}{Filtered Chebyshev}\\
& $E^s_{L^1}$& conv. rate&$E^s_{L^1}$&conv. rate\\
\midrule
$60$&$2.1928 \times 10^{-1}$&$-$&$3.2568\times 10^{-1}$&$-$
\\
$120$&$1.1380\times 10^{-1}$&$0.9463$&$8.2841\times 10^{-2}$&$1.9750$
\\
$600$&$2.4933\times 10^{-2}$&$0.9441$&$5.3537\times 10^{-3}$&$1.7697$
\\
$1200$&$1.6393\times10^{-2}$&$0.8831$&$1.0139\times 10^{-3}$&$1.8764$
\\
$6000$&$7.1933\times10^{-3}$&$0.7574$&$4.1364\times 10^{-4}$&$1.5131$
\\
\bottomrule
\end{tabular}
\renewcommand\arraystretch{1}
\caption{With reference to Section~\ref{sec:validation}, the relative error between the exact solution and its numerical approximation at time $t^s=1/6$ as function of the number of collocation points, both for the implementation of the finite volume scheme and the filtered Chebyshev method. The parameters for the simulations are $\Delta t=10^{-4}$ and $p=10$. The relative error is related to the incoming edges.}
\label{tab:error}
\end{table}

\subsection{Computation of the execution time for a Riemann problem at the junction}
\label{sec:CPU}
 We analyze here the performance of the method applied to a Riemann problem on network in terms of the computational cost required to complete the simulation and compare the results with the ones obtained by using finite volume schemes.

On each edge we consider the same flux $f_h(u)=u\left(1-u\right)$ and fix constant initial conditions:
\[
u_{1,0}(x) = \frac{1}{4},\quad u_{2,0}(x) =\frac{2}{3},\quad u_{3,0}(x) = \frac{4}{5}.
\]

In Table~\ref{tab:CPU} we find that, even if the method allows us to have a better accuracy with respect to the implementation of (FVS), it takes much more time to complete the simulation. The reason of such result could be related to the fact that, although FFT allows us to reduce the computation cost to obtain Chebyshev coefficients, the method has to take into account the cost of multiplications in the right hand side of~\eqref{eq:filterdiscmethodinterior}.

\begin{table}%
\centering%
\renewcommand\arraystretch{1.3}
\begin{tabular}{ccc
}
\toprule
\multirow{2}*{$N$} & \multicolumn{2}{c}{CPU time $[s]$ }\\
& FVS& Filtered Chebyshev\\
\midrule
$60$&$2.1587 \times 10^{0}$&$1.9512 \times 10^{0}$\\
$120$&$3.5127\times 10^{0}$&$3.6512\times 10^0$\\
$600$&$7.2146\times 10^{1}$&$9.324\times  10^2$\\
$1200$&$2.6903\times10^{3}$&$9.2543\times 10^{4}$
\\
$6000$&$3.2498\times10^{4}$&$2.3416\times 10^{5}$
\\
\bottomrule
\end{tabular}
\renewcommand\arraystretch{1}
\caption{With reference to Section~\ref{sec:CPU}, the CPU cost required by the (FVS) and filtered Chebyshev method, respectively, to complete the simulation on the whole network on the time interval $[0,1]$ as function of the number of collocation points $N$ on each road. The parameters for the simulations are $\Delta t=10^{-4}$ and $p=10$.}
\label{tab:CPU}
\end{table}

\subsection{Implementation of the 2D-Chebyshev spectral method}
\label{sec:2dchebycomp}

In this section, we perform a simulation on a degenerate network consisting in one single edge without any junction. This particular situation does not require the computation of the boundary data to establish the transmission condition, thus represents a good framework in order to validate the 2D-Chebyshev method described in~\ref{sec:fully-2D}.

We consider the same flux as in the previous simulations and we fix $[-1,1]$ as time domain, so that the initial condition has to be considered at time $t=-1$. We take
\[
u_0(x)=\chi_{[-3/4,-1/4]}(x) -\frac{2}{3}\left(x-\frac{5}{2}\right)\chi_{[-1/4,1/2]}(x)
\]
as initial condition.

In Figure~\ref{fig:comp-chebyxt}, we make a comparison between the profile of the solution at time $t=1$ computed by means of the filtered Chebyshev method and the 2D-Chebyshev method. As expected, We can observe a perfect agreement in the results.

A difference in the performance can be found in terms of execution time. Indeed, midpoint scheme seems more competitive, as shown in Table~\ref{tab:chebyxt}.

\begin{figure}
    \centering
    \includegraphics[width=0.5\textwidth]{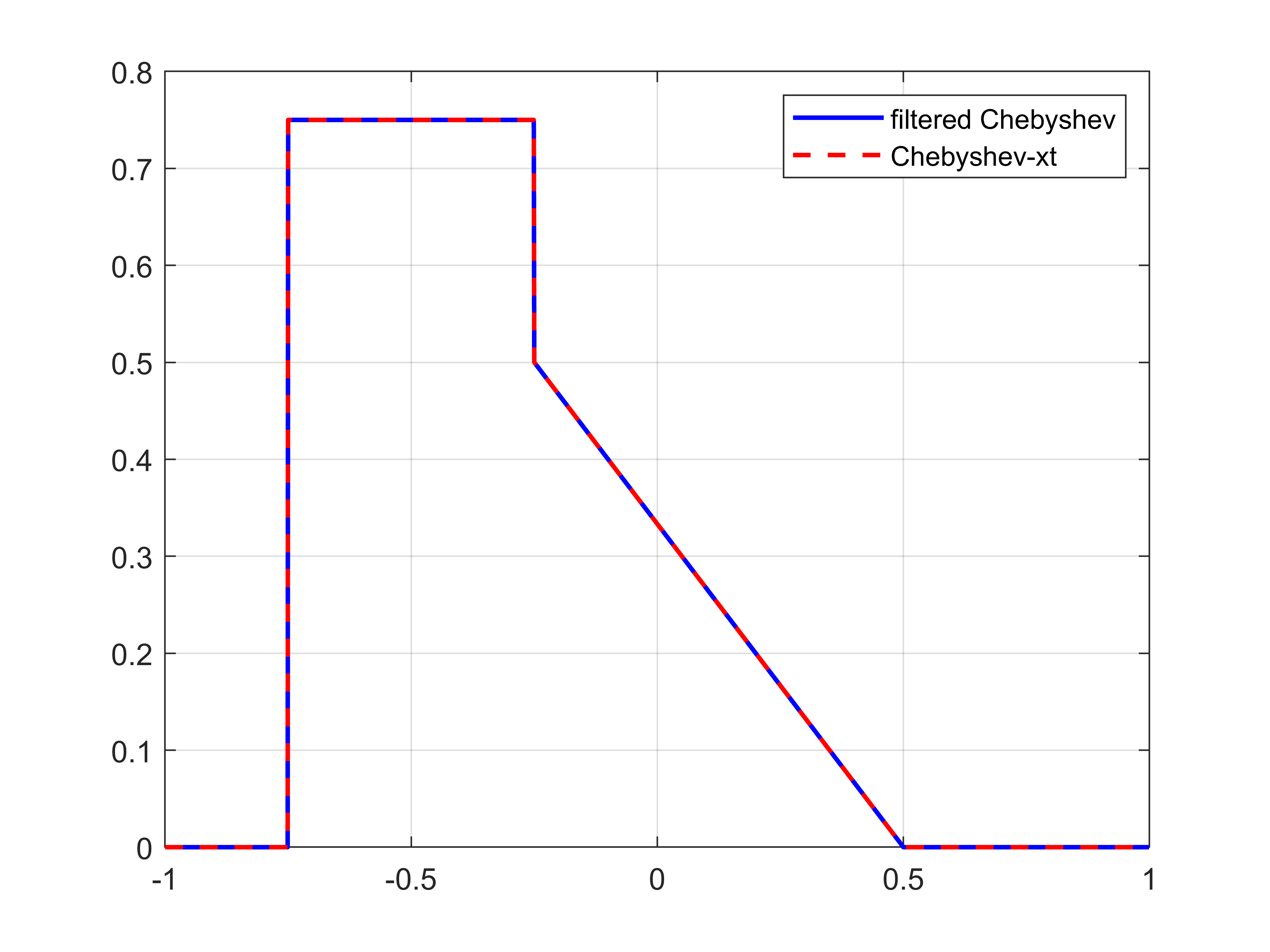}
    \caption{With reference to Section~\ref{sec:2dchebycomp}, the comparison of the solution at time $t=1$, obtained by implementing the filtered Chebyshev method coupled with the Midpoint method and the 2D-Chebyshev method for both spatial and time discretization.}
    \label{fig:comp-chebyxt}
\end{figure}

\begin{table}%
\centering%
\renewcommand\arraystretch{1.3}
\begin{tabular}{ccc
}
\toprule
\multirow{2}*{$N$} & \multicolumn{2}{c}{CPU time $[s]$ }\\
& Filtered Chebyshev& 2D Chebyshev\\
\midrule
$60$&$0.7163 \times 10^{0}$&$1.3651 \times 10^{0}$\\
$120$&$1.1753\times 10^{0}$&$4.3652\times 10^1$\\
$600$&$2.4087\times 10^{1}$&$7.2314\times  10^2$\\
$1200$&$8.952\times10^{2}$&$6.0697\times 10^{4}$
\\
$6000$&$1.0943\times10^{4}$&$3.1936\times 10^{5}$
\\
\bottomrule
\end{tabular}
\renewcommand\arraystretch{1}
\caption{With reference to Section~\ref{sec:2dchebycomp}, the CPU cost required by the filtered Chebyshev method coupled with Midpoint method and the 2D Chebyshev spectral scheme, respectively, as function of $N$.}
\label{tab:chebyxt}
\end{table}

\section{Conclusions}

The paper focuses on the implementation of a filtered Chebyshev spectral scheme to solve numerically a system of conservation laws on network. The modification of the Chebyshev coefficients by multiplication with a filter function allows us to avoid the Gibbs phenomenon near shock discontinuity. As a consequence, the method can benefit of the high-accuracy properties of spectral methods. We prove the convergence of the semi-discretized method and we perform some simulations in order to make a comparison between the proposed scheme and the implementation of a finite volume scheme. Our tests show the gain of one order in the convergence rate, even if the filtered Chebyshev method requires an higher computational cost. We also investigate the performance of a Chebyshev method to discretize the problem both in the space and in the time variable. This approach is applied only to a single edge is able to approximate well the solution of the problem, but it seems expensive from a computational point of view. Moreover, it need to be coupled with a different Riemann solver, because it requires to know in advance the values of the boundary data for every time. This issue represents the starting point for a further investigation.
\label{sec:concl}

\section*{Acknowledgments}
The author is member of the INdAM Research group GNCS. She has been supported by \textit{REFIN} Project, grant number D1AB726C funded by Regione Puglia, by INdAM - GNCS 2023 Project, grant number CUP$\_$E53C22001930001 and by \textit{PNRR MUR - M4C2}, grant number N00000013 - CUP D93C22000430001. She also acknowledges the partial support of ``Finanziamento giovani ricercatori 2022'' funded by GNCS-INdAM.

%

\bibliographystyle{plain}
\bibliography{biblio}

\end{document}